\documentclass[11pt]{article}
\usepackage{amssymb,amsmath}
\usepackage{amsthm}
\usepackage{amscd}
\newtheorem{theorem}{Theorem}[section]
\newtheorem{theorem*}{Theorem A\!\!}
\newtheorem{proposition}{Proposition}[section]
\newtheorem{proposition*}{Proposition A\!\!}

\newtheorem{corollary*}{Corollary A\!\!}

\newtheorem{lemma}{Lemma}[section]

\DeclareMathOperator{\sgn}{sgn}

\begin{document}
\title {Conformally covariant differential operators for the diagonal action of $O(p,q)$ on real quadrics}

\author{Jean-Louis Clerc}
\date{ }
\maketitle
\begin{abstract} Let $X=G/P$ be a real projective quadric, where $G=O(p,q)$ and $P$ is a parabolic subgroup of $G$. Let $\left(\pi_{\lambda,\epsilon}, \mathcal H_{\lambda,\epsilon}\right)_{ (\lambda,\epsilon)\in \mathbb C\times \{\pm\}}$ be the family of (smooth) representations of $G$ induced from the characters of $P$. For $(\lambda, \epsilon), (\mu, \eta)\in \mathbb C\times \{\pm\}$,  a differential operator $\mathbf D_{(\lambda,\epsilon), (\mu,\eta)}^{reg}$ on $X\times X$, acting $G$-covariantly from  $\mathcal H_{\lambda,\epsilon} \otimes \mathcal H_{\mu, \eta}$ into 
$\mathcal H_{\lambda+1,-\epsilon} \otimes \mathcal H_{\mu+1, -\eta}$ is constructed.
\end{abstract}
\section*{Introduction}

Let $S=S^n$ be the sphere of dimension $n$, equipped with its standard Riemannian structure. The group $G=O(1,n+1)$ acts conformally on $S$. For $\lambda\in \mathbb C$, let 
\[\mathcal H_\lambda=\{ f(x) (dx)^{\frac{\lambda}{n}},\quad f\in C^\infty(S)\}\] 
be the space of smooth $\frac{\lambda}{n}$-densities. The space $C^\infty(S)$  correspond to $\lambda=0$, whereas the space of measures on $S$ having a smooth density with respect to the Lebesgue measure $dx$ on $S$ corresponds to $\lambda = n$. The natural action of $G$ on $\mathcal H_\lambda$ induces a (smooth) representation $\pi_\lambda$ of $G$ on $\mathcal H_\lambda$. The family $(\pi_\lambda)_{\lambda\in \mathbb C}$ is known in semisimple harmonic analysis as the \emph{scalar principal series} of representations of $G$.

Now let $G$ act diagonally on $S\times S$. The tensor product $\pi_\lambda\otimes \pi_\mu$ of two representations of the principal series has a natural realization on  a space $\mathcal H_{\lambda, \mu}$ of sections of a certain line bundle over $S\times S$. In \cite{bc}, R. Beckmann and the present author constructed a family of differential operators on $S\times S$, depending on two complex parameters $(\lambda, \mu)$, which are covariant with respect to $(\pi_\lambda\otimes \pi_\mu, \pi_{\lambda+1}\otimes \pi_{\mu+1})$. The construction of these operators uses the heavy machinery of Knapp-Stein intertwining operators (see \cite{kn} for a general presentation). Whereas the covariance property of the operators is intrinsic to their definition, the fact that they are \emph{differential operators} is much more involved.  The problem is transferred (by using a stereographic projection) to the \emph{non-compact picture} or \emph{flat model} $\mathbb R^n\times \mathbb R^n$, and is solved through a long computation, using the Fourier transform on $\mathbb R^n$. See also \cite{c} Section 11 for a slightly different presentation of these results. This procedure was generalized recently to the geometric framework of \emph{completion of  simple real Jordan algebras} (see \cite{bck}).

The present paper gives a more elementary construction of these operators in the geometric setting of the real quadrics. The philosophy behind the present construction is based on the following observation. Let $X$ be a real quadric, and let $G$ be its group of conformal transformations. Then $G$ has an open dense orbit in  its diagonal action on $X\times X$ which is a \emph{reductive symmetric space} (see Proposition \ref{openorbit} for a more explicit statement). This rich underlying geometric structure explains that it is easy to construct $G$-covariant differential operators on this open orbit. The next question is to study wether such a differential operator can be smoothly extended to $X\times X$. 

The real quadric $X$ is realized as the projective variety associated to the isotropic  cone $\Xi$ of the \emph{ambient space} $(V,Q)$, where $Q$ is a quadratic form on a real vector $V$. The group of conformal transformations of $X$ is $O(Q)$, acting projectively on $X$. To construct covariant differential operators on $X$ (or on $X\times X$), it is wise to start with a homogenous $G$-invariant differential operator on $V$ (or $V\times V$) and try to induce a differential operator on $\Xi$ (or $\Xi\times \Xi$). This is possible only if the operator on $V$ is \textquotedblleft tangential along  $\Xi$\textquotedblright (or along $\Xi\times \Xi)$. The corresponding verification is obtained through computations in the Weyl algebra ($=$ algebra of differential operators with polynomial coefficients) of $V$. 

To finish this introduction, let us mention an application of these operators, which is not developed in this article. By restriction to the diagonal, they provide \emph{covariant bi-differential operators} from $X\times X$ to  $X$. As it is possible to compose (appropriate) covariant differential operators on $X\times X$, the restriction process also yields higher order covariant bi-differential operators. These bi-differential operators are generalizations of the classical \emph{Rankin-Cohen brackets} (see \cite{bc} Theorem 3.4 or \cite{bck} Section 8). A similar approach for Juhl's conformally covariant differential operators from $S^n$ to $S^{n-1}$ was proposed in \cite{c2}.

\section{The real quadric and a series of representations of $O(p,q)$}

Let $V$ be a real vector space of dimension $n=p+q$ where $p,q$ are natural integers such that $p,q\geq 1, p+q\geq 3$
, and let $Q$ be a quadratic form on $V$ of signature  $(p,q)$. Choose a basis $e_1,e_2,\dots, e_p,e_{p+1},\dots, e_n$ such that the quadratic form $Q$ is given by
\[Q(\mathbf v) = Q(x_1,x_2,\dots, x_p,x_{p+1},\dots, x_n) = x_1^2+x_2^2+\dots +x_p^2-x_{p+1}^2-\dots -x_n^2\ .
\]
The corresponding symmetric bilinear form will be denoted also by $Q$, namely for $\mathbf v = (x_1,x\dots,x_p,x_{p+1},\dots, x_n)$ and $\mathbf w = (y_1,\dots, y_p,y_{p+1},\dots, y_n)$
\[Q(\mathbf v, \mathbf w) = x_1y_1+\dots+x_py_p-x_{p+1}y_{p+1}-\dots-x_ny_n\ .
\]

For $\mathbf v\in V, \mathbf v\neq 0$, let $[\mathbf v]= \mathbb R^*\mathbf  v$ be its corresponding element in the projective space $\mathbb P(V)$.

Consider the \emph{proper isotropic cone} \[\Xi = \{ \mathbf v\in V, \mathbf v\neq 0,\quad Q(\mathbf v) = 0\}\ .\]
For $\mathbf v\neq 0$, the differential $dQ(\mathbf v)=2Q(\mathbf v, .)$ is $\neq 0$ and hence $Q=0$ is a regular equation of $\Xi$ near any point of $\Xi$. The projective quotient  $X=\Xi/\mathbb R^*$ is a real quadric.

The group $G=O(Q)\simeq O(p,q)$ preserves $\Xi$. As the action of $G$ commutes with the dilations, the group $G$ acts naturally on $X$. As a consequence of Witt theorem, this action is transitive.

 An open subset $\mathcal O$ (resp. $\Omega$) of $V\smallsetminus\{0\}$ (resp. $\Xi$) is said to be \emph{conical} if $\mathcal O$ (resp. $\Omega$) is stable by all dilations $\mathbf v\mapsto r\mathbf v, r\in \mathbb R^*$.
 
 For $\lambda\in \mathbb C, \epsilon\in \{ \pm\}$ and for $r\in \mathbb R^*$, let  
\[r^{\lambda, \epsilon} = \left\{\begin{matrix} \vert r\vert^\lambda& \text{ if }\epsilon = +\\\sgn(r)\vert r\vert^\lambda& \text{ if } \epsilon = -\end{matrix}\right.\qquad .
\]
 Let $\mathcal O$ be a conical open subset of $V\smallsetminus \{0\}$, and let $(\lambda, \epsilon) \in \mathbb C\times \{\pm\}$. Set
 \[\mathcal F_{\lambda, \epsilon}(\mathcal O) = \{ F\in C^\infty(\mathcal O), F(r\mathbf v) = r^{-\lambda, \epsilon} F(\mathbf v)\quad \forall r\in \mathbb R^*,\mathbf v \in \mathcal O\}\ .
 \]
Similarly, for $\Omega$ a conical open subset of $\Xi$, let
 \[\mathcal H_{\lambda, \epsilon}(\Omega)=\{ F\in C^\infty(\Omega), F(r\mathbf v) = r^{-\lambda, \epsilon}F(\mathbf v),\quad \forall r\in \mathbb R^*,\mathbf v \in \Omega\}, \]
 and  simply let $\mathcal H_{\lambda, \epsilon} = \mathcal H_{\lambda, \epsilon}(\Xi)$, equipped with its natural Fr\'echet topology.
 
  For $g\in G$, and $F\in \mathcal H_{\lambda, \epsilon}$, let
 \[\pi_{\lambda, \epsilon}(g) F = F\circ g^{-1}\ .
 \]
 Then $\pi_{\lambda, \epsilon}(g) F$ belongs to $\mathcal H_{\lambda, \epsilon}$ and  this defines  a (smooth) representation $\pi_{\lambda, \epsilon}$ of $G$ on $\mathcal H_{\lambda, \epsilon}$.

Homogenous functions on $\Xi$ are interpreted as sections of a corresponding line bundle on $X$, and conversely, differential operators for these line bundles over $X$ are viewed as  differential operators acting on homogenous functions on $\Xi$. These identifications are tacitly used in the sequel.

 \section{The covariant differential operator $\widetilde \square$}
 
 Let $q\in \mathbb C[V]$ be a polynomial on $V$. There is a unique constant coefficients  differential operator, denoted by $\displaystyle q\left(\frac{\partial}{\partial \mathbf x}\right)$ such that for any $\mathbf y\in V$
 \[ q\left(\frac{\partial}{\partial \mathbf x}\right) e^{Q(\mathbf x, \mathbf y)} = q(\mathbf y) e^{Q(\mathbf x, \mathbf y)}\ .
 \]
The operator $\displaystyle q\left(\frac{\partial}{\partial \mathbf x}\right)$ is $G$-invariant (i.e. commutes with the action of $G$) if and only if $q$ is a $G$-invariant polynomial on $V$. Choosing $q=Q$, this yields the \emph{d'Alembertian operator}
$\displaystyle \square = Q\left(\frac{\partial}{\partial \mathbf x}\right)$, which in the coordinates $(x_1,x_2,\dots, x_n)$ reads
 \[\square = \frac{\partial^2}{\partial x_1^2}+\dots+ \frac{\partial^2}{\partial x_p^2}- \frac{\partial^2}{\partial x_{p+1}^2}-\dots - \frac{\partial^2}{\partial x_n^2}\ .
 \]  
 The \emph{Weyl algebra} is the algebra of differential operators on $V$ having polynomial coefficients. For $q\in \mathbb C[V]$, the multiplication operator by $q$ is simply denoted by $q$ or $q(\mathbf x)$, depending on the context. The composition of operators in the Weyl algebra is usually denoted by $\circ$. However when multiplication by a polynomial is performed \emph{after} a constant coefficient differential operators, the symbol $\circ$ may be omitted.
 
The construction of covariant differential operators on the quadric $X$ is well-known (see e.g. \cite{es}) and is recalled here, as it is used and serves as a model for the more elaborate  constructions to come.
\begin{lemma}\label{SQG}
  \[\square\circ Q = 2n+4\mathbf E  +Q\,\square ,
 \]
 where $\mathbf E$ is the \emph{Euler operator} given by
 $\displaystyle\mathbf E = \mathbf E\left(\mathbf x, \frac{\partial}{\partial \mathbf x} \right)=\sum_{j=1}^n x_j\frac{\partial}{\partial x_j}$.
 
 \end{lemma}
\begin{proof} 
Straightforward computation.
\end{proof}
 Let $F\in \mathcal H_{\lambda, \epsilon}$. It is possible to extend $F$  to a function $\overline F \in \mathcal F_{\lambda, \epsilon}(\mathcal O)$ for $\mathcal O$ a conical neighborhood of  $\Xi$ in $V\smallsetminus \{0\}$. The restriction of $\square \overline{F}$ to $\Xi$ belongs to $\mathcal H_{\lambda+2,\epsilon}$. However, the extension is not unique and the  restriction of $\square \overline F$ to $\Xi$ usually depends of the extension.

\begin{proposition}\label{sqxi}
 Let $F\in \mathcal H_{\frac{n}{2}-2, \epsilon}$. Let $\overline F$ an extension of $F$ to a conical neighborhood of $\Xi$ as above.Then the restriction of $\square \overline{F}$ to $\Xi$ only depends on the values of $F$ on $\Xi$.
 \end{proposition}
 
 \begin{proof} It is enough to show that if $\overline F$ vanishes on $\Xi$, then $\square \overline F$ vanishes on $\Xi$. But such a function can be written as $\overline F=QG$, where $G$ is defined in a conical neighborhood $\mathcal O$ of $\Xi$ and satisfies $G(r\mathbf x) = r^{-\frac{n}{2},\epsilon} G(\mathbf x)$ for all $r\in \mathbb R^*$ and $\mathbf x\in \mathcal O$. But differentiating this relation at $r=1$ yields
 $\mathbf E G = -\frac{n}{2} G$. Hence by Lemma \ref{SQG} 
$\square(QG) = Q\square G
$, which implies for $\mathbf x\in \Xi$\[\square F(\mathbf x) = Q(\mathbf x)\, \square G(\mathbf x) = 0\ .\]
 \end{proof} 
Proposition \ref{sqxi} defines a differential operator on $X$
 \[\widetilde \square : \qquad F\longmapsto \overline F \longmapsto \square \overline F\longmapsto \square \overline F_{\vert \Xi} ,\] 
mapping $\mathcal H_{\frac{n}{2}-2, \epsilon}$ into $\mathcal H_{\frac{n}{2}, \epsilon}$. Moreover, as $\square$ commutes with the natural action of $G$ on  functions, the operator $\widetilde \square$ intertwines $\pi_ {\frac{n}{2}-2, \epsilon}$ and $\pi_{\frac{n}{2}, \epsilon}$.

\section{The operators $\mathbf D_{(\lambda,\epsilon),(\mu,\eta)}$ on $(X\times X)^\times$.}

Let 
\[(V\times V)^\times=\{(\mathbf x, \mathbf y)\in V\times V,\quad Q(\mathbf x,\mathbf y)\neq 0\}.
\]
Clearly, $(V\times V)^\times$ is a conical dense open subset of $V\times V$, 
which is invariant under the diagonal action of $G$ on $V\times V$. Similarly, let
\[(\Xi\times \Xi)^\times = \{ (\mathbf x, \mathbf y)\in \Xi\times \Xi, \quad Q(\mathbf x, \mathbf y)\neq 0\}\ .
\]
Consider the corresponding projective situation, i.e. let
\[(X\times X)^\times=  (\Xi\times \Xi)^\times/(\mathbb R^*\times \mathbb R^*) \ .\]
The diagonal action of $G$ on $(V\times V)^\times$ induces an action  on $(X\times X)^\times$.
The next proposition will not be used in the sequel, but, as commented in the introduction, it is in the background of the construction of the covariant differential operators on $X\times X$.

\begin{proposition}\label{openorbit}
 The group $G$ acts transitively on $(X\times X)^\times$ and the stabilizer of a generic element of $(X\times X)^\times$ is a subgroup of index 2 in a reductive symmetric subgroup of $G$.
\end{proposition}
\begin{proof}
Let $(\mathbf x ,\mathbf y)$ be a representative of an element in $(X\times X)^\times$. Without loosing any generality, it is possible to assume that $Q(\mathbf x, \mathbf y)=1$. Similarly, let $[\mathbf x', \mathbf y']$ be a representative of another element of $(X\times X)^\times$, such that $Q(\mathbf x', \mathbf y')=1$. Recall that $Q(\mathbf x) = Q(\mathbf x')=0$ and $Q(\mathbf y)=Q(\mathbf y')=0$. By Witt theorem, there exists an isometry $g$ of $(V,Q)$ such that $g(\mathbf x)=\mathbf x', g(\mathbf y)= \mathbf y'$, thus proving the first part of the proposition

Next, let $(\mathbf x,\mathbf y)\in (V\times V)^\times$. The restriction of $Q$ to the 2-subspace $\mathbb R\mathbf x\oplus \mathbb R \mathbf y$ is of signature $(1,-1)$. Hence 
\[V=(\mathbb R\mathbf x\oplus \mathbb R \mathbf y) \oplus (\mathbb R\mathbf x\oplus \mathbb R \mathbf y)^\perp\ ,
\]
and the restriction of $Q$ to $(\mathbb R\mathbf x\oplus \mathbb R \mathbf y)^\perp$ is of signature $(p-1,q-1)$. Let $\sigma$ be the transformation which is $+1$ on $\mathbb R\mathbf x\oplus \mathbb R \mathbf y$ and $-1$ on $(\mathbb R\mathbf x\oplus \mathbb R \mathbf y)^\perp$. If $g\in G$ stabilizes both $[\mathbf x]$ and $[\mathbf y]$, then $g$ stabilizes $\mathbb R\mathbf x\oplus \mathbb R \mathbf y$ and its orthogonal subspace $(\mathbb R\mathbf x\oplus \mathbb R \mathbf y)^\perp$, so that $\sigma \circ g\circ \sigma=  g$. Let \[
H = \{ g\in G, \sigma \circ g\circ \sigma\} \ .\]
Then $H\simeq O(1,-1)\times O(p-1,q-1)$ is a symmetric reductive subgroup of $G$, and the stabiliser $G^{[\mathbf x], [\mathbf y]}$ of $([\mathbf x],[\mathbf y]$ in $G$ is the subgroup of $H$ of index $2$, isomorphic to $\mathbb R^*\times O(p-1,q-1)$.
\end{proof}
Let $(\lambda, \epsilon),(\mu,\eta) \in \mathbb C\times \{\pm\}$. For $\mathcal O$ a conical open set of $(V\times V)^\times$, let
\[\mathcal F_{(\lambda, \epsilon)(\mu, \eta)}(\mathcal O) = \{ F\in C^\infty(\mathcal O),\quad F(r\mathbf x, s\mathbf y) = r^{-\lambda, \epsilon}s^{-\mu, \eta} F(\mathbf x, \mathbf y)\}
\]
for all $(\mathbf x, \mathbf y)\in \mathcal O$ and $r,s\in \mathbb R^*$.
Similarly, for $\Omega$ a conical open subset of $(\Xi\times \Xi)^\times$ let $\mathcal H_{(\lambda, \epsilon)(\mu, \eta)}(\Omega) $ be the space of all functions $ F\in C^\infty(\Omega)$ such that 
\[F(r\mathbf x, s\mathbf y) = r^{-\lambda, \epsilon}s^{-\mu, \eta} F(\mathbf x, \mathbf y),\quad  \text {for all }(\mathbf x, \mathbf y)\in \Omega \text{ and } r,s\in \mathbb R^*\ .
\]

The space corresponding to $\Omega = (\Xi\times \Xi)^\times$ is denoted by
$\mathcal H_{(\lambda, \epsilon), (\mu,\eta)}^\times$. The diagonal action of $G$ on $\Xi\times \Xi$ induces a representation of $G$ on $\mathcal H_{(\lambda, \epsilon), (\mu, \eta)}^\times$.

Consider the differential operator on $(V\times V)^\times$ given by
\[\mathbf E_{\lambda, \mu} =\sgn(Q(\mathbf x, \mathbf y)) \quad \times\]\[\vert Q(\mathbf x,\mathbf y)\vert^{+\frac{n}{2}-\mu-1}\circ \square\left(\frac{\partial}{\partial \mathbf y}\right)\circ \vert Q(\mathbf x,\mathbf y)\vert^{-\lambda+\mu} \circ \square\left(\frac{\partial}{\partial \mathbf x}\right)\circ \vert Q(\mathbf x, \mathbf y)\vert^{-\frac{n}{2}+2+\lambda}\ .
\]
The operator $\mathbf E_{\lambda, \mu}$ is well defined on $C^\infty\big((V\times V)^\times\big)$, and commutes with the diagonal action of $G$ on $(V\times V)^\times$. Let $F\in \mathcal H_{(\lambda, \epsilon), (\mu, \eta)}^\times$. Extend it to a function $\overline F\in \mathcal F_{(\lambda, \epsilon),(\mu, \eta)}(\mathcal O)$ where $\mathcal O$ is conical neigborhood of $(\Xi\times \Xi)^\times$ in $(V\times V)^\times$.
For $\mathbf y$ fixed, the function
\[G_{\mathbf y} : \mathbf x\longmapsto \vert Q(\mathbf x, \mathbf y)\vert^{-\frac{n}{2}+2+\lambda}
\overline F(\mathbf x, \mathbf y)\] is defined and smooth on a conical neigborhood of
$\Xi_{\mathbf y} = \{ \mathbf x\in \Xi, Q(\mathbf x, \mathbf y)\neq 0\}$ and homogenous of degree $-\frac{n}{2}+2$. Hence, by (a localized version of) Proposition \ref{sqxi}, the restriction  to $\Xi_{\mathbf y}$ of
 $\square\left(\frac{\partial}{\partial \mathbf x}\right) G_{\mathbf y}$  depends only on the values of $\overline F$ on $\Xi_{\mathbf y} $.
 
 Similarly, for $\mathbf x$ fixed the function
 \[H_{\mathbf x} : \mathbf y\longmapsto \vert Q(\mathbf x,\mathbf y)\vert^{-\lambda+\mu}\, \,\square\left(\frac{\partial}{\partial \mathbf x}\right)\left(\vert Q(\mathbf x, \mathbf y)\vert^{-\frac{n}{2}+2+\lambda}\overline F(\mathbf x, \mathbf y)\right)
 \]
is defined and smooth on a conical neighborhood of $\Xi_{\mathbf x}$ and homogeneous of degree $-\frac{n}{2}+2$. Hence, by the same argument as above, the restriction to $\Xi_{\mathbf x}$ of the function $\square\left(\frac{\partial}{\partial \mathbf y}\right)H_\mathbf x$ depends only of the values of $\overline F$ on $\Xi_{\mathbf x}$.

These observations and some elementary verifications about the homogeneity and the action of $G$ yields the following proposition. 
\begin{proposition}\label{Elambdamutan}
 The operator $\mathbf E_{\lambda, \mu}$ induces a differential operator
 \[\mathbf D_{(\lambda,\epsilon),( \mu,\eta)} :  \mathcal H_{(\lambda, \epsilon),(\mu, \eta)}^\times \longrightarrow  \mathcal H_{(\lambda+1, -\epsilon),(\mu+1, -\eta)}^\times\ .\]
 The induced operator commutes with the natural actions of $G$ on each of the fonction spaces involved.
\end{proposition}

To have a better understanding of the behavior of the operator $\mathbf E_{\lambda, \mu}$ near the singular set where $Q(\mathbf x, \mathbf y) =0$, a more explicit expression of the operator $\mathbf E_{\lambda, \mu}$ is needed.

\begin{proposition}\label{explicitElambdamu}
 The following identity holds on $(V\times V)^\times$
\begin{equation*}
\begin{split}
\mathbf E_{\lambda, \mu}  =\hskip 6cm &\\
(-\frac{n}{2} +2+\lambda) (-\frac{n}{2} +1+\lambda)(-\frac{n}{2} +\mu)(-\frac{n}{2} +\mu-1) Q(\mathbf x, \mathbf y)^{ -3} Q(\mathbf x)Q(\mathbf y)\quad &
(I)\\
+ 2(-\frac{n}{2} +2+\lambda) (-\frac{n}{2} +1+\lambda)Q(\mathbf x,\mathbf y)^{-2} Q(\mathbf y) \circ \left( \sum_{j=1}^n x_j\frac{\partial }{\partial y_j}\right)\hskip 1cm& (II)\\
+ (-\frac{n}{2} +2+\lambda) (-\frac{n}{2} +1+\lambda)(2n-4+4\mu)\ Q(\mathbf x,\mathbf y)^{-1} \hskip2cm &(III)\\
+ 4(-\frac{n}{2} +2+\lambda) (-\frac{n}{2} +1+\lambda)\ Q(\mathbf x,\mathbf y)^{-1} \circ \left( \sum_{j=1}^n y_j\frac{\partial }{\partial y_j}\right)\hskip1.5cm & (IV)\\
+(-\frac{n}{2} +2+\lambda) (-\frac{n}{2} +1+\lambda)Q(\mathbf x, \mathbf y)^{-1} Q(\mathbf y) \,Q\left(\frac{\partial }{\partial \mathbf y} \right)\hskip 2cm& (V)\\
+2(-\frac{n}{2}+2+\lambda)(-\frac{n}{2}+1+\mu)(-\frac{n}{2}+\mu)\ Q(\mathbf x) Q(\mathbf x, \mathbf y)^{-2}\left(\sum_{j=1}^n y_j \frac{ \partial }{\partial x_j} \right)\quad & (VI)\\
+ 2(-\frac{n}{2}+2+\lambda)(-\frac{n}{2}+1+\mu)Q(\mathbf x, \mathbf y)^{-1}
\left(\sum_{j=1}^n x_j\frac{\partial }{\partial x_j}\right) \hskip 1.7cm & (VII)\\
+ 2(-\frac{n}{2}+2+\lambda)(-\frac{n}{2}+1+\mu)\ Q(\mathbf x, \mathbf y)^{-1}\left(\sum_{j=1}^n\sum_{k=1}^n x_jy_k \frac{\partial^2 }{\partial x_ky_j}\right)\qquad &(VIII)\\
+2(-\frac{n}{2}+2+\lambda)\left( -2 \sum_{j=1}^n \frac{ \partial^2}{\partial x_j\partial y_j}+\big(\sum_{k=1}^n y_k\frac{\partial} {\partial x_k}\big)\circ Q\left(\frac{ \partial }{\partial \mathbf y}\right)   \right)\hskip 0.8cm & (IX)\\
+(-\frac{n}{2} +2+\mu)((-\frac{n}{2} +1+\mu)\ Q(\mathbf x, \mathbf y)^{-1} Q(\mathbf x)\ Q\left(\frac{\partial}{\partial \mathbf  x} \right)\hskip1.5cm & (X)\\
+ 2(-\frac{n}{2} +2+\mu) \left(\sum_{j=1}^n x_j \frac{\partial}{\partial y_j}\right) \circ Q\left( \frac{ \partial}{\partial \mathbf x}\right)\hskip 2.5cm & (XI)\\
+Q(\mathbf x, \mathbf y) Q\left(\frac{ \partial}{\partial \mathbf y}\right) Q\left(\frac{ \partial}{\partial \mathbf x}\right)\hskip4cm & (XII)\ .
\end{split}
\end{equation*}

\end{proposition}

\begin{proof}
Computations are first made on $\{(\mathbf x, \mathbf y)\in V\times V, Q(\mathbf x, \mathbf y)>0\}$ and it will be indicated at the end how to handle the situation when $Q(\mathbf x, \mathbf y)<0$. With this extra assumption, it is possible, for $\rho$ any complex number,  to replace  $\vert Q(\mathbf x, \mathbf y)\vert^{\rho}$ by simply $Q(\mathbf x, \mathbf y)^\rho$.  An intermediate calculation yields
\[Q(\mathbf x,\mathbf y)^{-\lambda+\mu} \circ Q\left(\frac{\partial}{\partial \mathbf x}\right)\circ Q(\mathbf x,\mathbf y)^{-\frac{n}{2}+2+\lambda}= 
\]
\[(-\frac{n}{2}+2+\lambda)(-\frac{n}{2}+1+\lambda) Q(\mathbf x,\mathbf y)^{-\frac{n}{2}+\mu}\circ Q(\mathbf y)\]
\[+2(-\frac{n}{2}+2+\lambda)Q(\mathbf x,\mathbf y)^{-\frac{n}{2}+1+\mu} \circ \big(\sum_{j=1}^n y_j \frac{\partial }{\partial x_j} \big)
\]
\[+Q(\mathbf x, \mathbf y)^{-\frac{n}{2}+2+\mu} \circ Q\left(\frac{\partial }{\partial \mathbf x}\right) \ .
\]
After a long but straightforward computation, the formula of Proposition \ref{explicitElambdamu} is obtained.

To finish the proof, it is enough to justify that the same formula is valid on the domain where $Q(\mathbf x, \mathbf y)<0$. To see this, let $Q' = -Q$ be the opposite quadratic form, and let $\mathbf E'_{\lambda, \mu}$ be the differential operator obtained from $Q'$ by the same procedure as for obtaining $\mathbf E_{\lambda, \mu}$ from $Q$. If $Q(\mathbf x, \mathbf y)<0$, $Q'(\mathbf x, \mathbf y)>0$, so that the previous computation can be used to evaluate $\mathbf E'_{\lambda, \mu}$ using $Q'$ instead of $Q$. Now each  term (from (I) to (XII) corresponding to the explicit expression of $\mathbf E'_{\lambda, \mu}$, can be rewritten using $Q=-Q'$. But each of the twelve terms labeled from $(I)$  to $(XII)$ is easily seen to be changed to its opposite when changing $Q'$ to $Q=-Q'$. The conclusion follows as it is easily seen directly from their definition that $\mathbf E'_{\lambda, \mu}=-\mathbf E_{\lambda, \mu}$.
\end{proof}.
\section{The operators $\mathbf D^{reg}_{(\lambda, \epsilon)(\mu,\eta)}$ on $X\times X$}
For reasons to be explicited later,  the term labeled  $(VIII)$ has to be written differently. Let
\[\mathbf E_{(VIII)} = Q(\mathbf x, \mathbf y)^{-1} \sum_{j=1}^n \sum_{k=1}^n x_jy_k \frac{\partial^2}{\partial x_k\partial y_j},
\]
and
\[\mathbf E \left(\mathbf x,\frac{\partial}{\partial \mathbf y}\right) =  \sum_{j=1}^n x_j\frac{\partial}{\partial y_j}
,\qquad \mathbf E\left(\mathbf y,\frac{\partial}{\partial \mathbf x} \right)=  \sum_{j=1}^n y_j\frac{\partial}{\partial x_j}\ .
\]
The following elementary commutation turns out to be the  key for the sequel.%
\begin{lemma}\label{commEVIII}
\begin{equation}
\left[\mathbf E_{(VIII)}\ ,\ Q(\mathbf x) \right]=2 \mathbf E\left(\mathbf x, \frac{\partial}{\partial \mathbf y}\right)\end{equation}

\begin{equation}
\left[Q\left(\frac{\partial}{\partial \mathbf x}, \frac{\partial}{\partial \mathbf y}\right)\ ,\  Q(\mathbf x)\right]=  2\mathbf E\left( \mathbf x, \frac{\partial}{\partial \mathbf y}\right)\ .
\end{equation}
\end{lemma}
Consider the differential operator 
\[\mathbf F =  \mathbf E_{(VIII)} -Q\left( \frac{\partial}{\partial \mathbf x}, \frac{\partial}{\partial \mathbf y}\right)
\]
\begin{proposition}\label{Ftan}
\begin{equation}
\mathbf F \circ Q(\mathbf x) = Q(\mathbf x)\circ \mathbf F,\qquad \mathbf F \circ Q(\mathbf y) = Q(\mathbf y)\circ \mathbf F\ .
\end{equation}
\end{proposition}

\begin{proof} Use Lemma \ref{commEVIII},  the second half of the statement being obtained by exchanging $\mathbf x$ and $\mathbf y$.
\end{proof}

\begin{proposition}\label{Ftan2}
 Let $f$ be a smooth function  on $(\Xi\times \Xi)^\times$. Let $\overline f$ be a smooth extension of $f$  to a conical neighborhood $\mathcal O$ of $(\Xi\times \Xi)^\times$ in $(V\times V)^\times$. The restriction of $\mathbf F( \overline f)$ to $(\Xi\times \Xi)^\times$  depends only on $f$ and not of the particular extension used.
\end{proposition}

\begin{proof} It is enough to prove that if $\overline f$ vanishes on $(\Xi\times \Xi)^\times$, then $\mathbf F (\overline f)$ vanishes on $(\Xi\times \Xi)^\times$. Now such a function $\overline f$ can be written (in a conical neighborhood $\mathcal O$ of a given ray $\mathbb R^* (\mathbf x_0, \mathbf y_0)\subset (\Xi\times \Xi)^\times$) as
\[{\overline f}(\mathbf x, \mathbf y) = Q(\mathbf x)g(\mathbf x, \mathbf y) + Q(\mathbf y) h(\mathbf x, \mathbf y)\ .
\]
where $g$ and $h$ are smooth functions on $\mathcal O$. From Lemma \ref{commEVIII} follows
\[\left[\mathbf F\ ,\  Q(\mathbf x)\right] = \left[\mathbf F\ ,\  Q(\mathbf y)\right] 
=0\]
and hence
\[\mathbf F (\overline f)= Q(\mathbf x) \mathbf F (g) + Q(\mathbf y) \mathbf F(h)\ .
\]
When $(\mathbf x, \mathbf y)$ belongs to $(\Xi\times \Xi)^\times$, then $Q(\mathbf x) = Q(\mathbf y) = 0$ and the proposition follows.
\end{proof}
Consider now the following decomposition 
\[\mathbf E_{\lambda, \mu} = \mathbf E_{\lambda, \mu}^{sing}+ \mathbf E_{\lambda, \mu} ^{reg}
\]
where
\[\mathbf E_{\lambda, \mu}^{sing} = (I)+\dots+(VII) + 2(-\frac{n}{2} +1+\lambda)(-\frac{n}{2}+\mu)\, \mathbf F +(X)
\]
\[\mathbf E_{\lambda, \mu} ^{reg} = 2(-\frac{n}{2} +1+\lambda)(-\frac{n}{2}+\mu)\ Q\left(\frac{\partial}{\partial \mathbf x}, \frac{\partial}{\partial \mathbf y}\right)+ (IX)+ (XI) + (XII)
\]
\begin{proposition}\label{Esingtan}
 Let $\mathcal O$ be a conical neighborhood of $(\Xi\times \Xi)^\times$, and let $f\in \mathcal F_{(\lambda, \epsilon),(\mu, \eta)}(\mathcal O)$. Then the restriction of $\mathbf E_{\lambda, \mu}^{reg} f$ to $(\Xi\times \Xi) ^{reg}$ only depends on the restriction of $f$ to $(\Xi\times \Xi)^{reg}$.
\end{proposition}

\begin{proof}From Proposition \ref{Elambdamutan}, it is equivalent to prove the similar statement for $\mathbf E_{\lambda,\mu}^{sing}$. As already argued above, it is suffisant to prove that if $f$ vanishes on $(\Xi\times \Xi)^\times$, then $\mathbf E_{\lambda,\mu}^{sing} f$ vanishes on $(\Xi\times \Xi)^\times$.  In the expression of $\mathbf E_{\lambda, \mu}^{sing}(f)$, terms corresponding to factors (I), (II), (V), (VI), (X) vanish on $(\Xi\times \Xi)^{reg}$ as they contain either a factor $Q(\mathbf x)$ or $Q(\mathbf y)$. Terms (III), (IV) and (VII) when evaluated on $f$ vanish on $(\Xi\times \Xi)^{reg}$. This is trivially true  for term (III). For terms (IV) and (VII), use the homogeneity condition of $f$ to justify the statement. The extra term in the definition of $\mathbf E_{\lambda, \mu}^{sing}$ is proportional to $\mathbf F$. But $\mathbf F(f)$ vanishes on $(\Xi\times \Xi)^\times$  as a consequence of Proposition \ref{Ftan2}. This achieves the proof.
\end{proof}
\begin{theorem} The operator $\mathbf E_{\lambda, \mu}^{reg}$ induces a map from $\mathcal H_{(\lambda, \epsilon), (\mu, \eta)}$ into \break $\mathcal H_{(\lambda+1,-\epsilon), (\mu+1,- \eta)}$.
Viewed as an operator on sections of  line bundles over $\Xi\times \Xi$, it is a differential operator $\mathbf D^{reg}_{(\lambda,\epsilon), (\mu,\eta)}$ which intertwines the representations $\pi_{\lambda, \epsilon} \otimes \pi_{\mu,\eta}$ and $\pi_{\lambda+1,-\epsilon} \otimes \pi_{\mu+1,-\eta}$.
\end{theorem}

\begin{proof} Let $f\in \mathcal H_{(\lambda,\epsilon), (\mu,\eta)}$, extend it, respecting the homogeneities, to a neighborhood of $\Xi\times \Xi$, still denoted by $f$. By Proposition 
\ref{Esingtan} the value of $\mathbf E_{\lambda, \mu}^{reg}f$ on $(\Xi\times \Xi)^{reg}$ depends only on the values of $f$ on $(\Xi\times \Xi)^{reg}$. The operator $\mathbf E_{\lambda,\mu}^{reg}$ has polynomial coefficients on $V\times V$. Hence by continuity, the value of $\mathbf E_{\lambda,\mu}^{reg}f$ at a point in $(\Xi\times \Xi)$ is well defined and depends only on the values of $f$ on $\Xi\times \Xi$. The invariance property of $\mathbf E_{\lambda, \mu}^{reg}$ with respect to the diagonal action of $G$ on $V\times V$ follows immediately from the definition of $\mathbf E_{\lambda, \mu}^{reg}$ and implies the intertwining relation for $\mathbf D_{(\lambda,\epsilon), (\mu,\eta)}^{reg}$.
\end{proof}

Now notice that  the term (IX)  can be rewritten using 
\begin{equation}\label{modifIX}
\sum_{k=1}^n y_n \frac{\partial}{\partial x_n} =  \sum_{k=1}^n (y_k-x_k) \frac{\partial}{\partial x_k} + \mathbf E\left(\mathbf x, \frac{\partial}{\partial \mathbf x}\right)\ .
\end{equation}
A similar modification is possible for  the term (XI). 

Define the operator 
\[\mathbf F_{\lambda, \mu} = \] 
\[Q(\mathbf x, \mathbf y)\,\circ Q\left(\frac{ \partial}{\partial \mathbf y}\right)\circ  Q\left(\frac{ \partial}{\partial \mathbf x}\right)
\]
\[+2(-\frac{n}{2} +1+\mu)  \left(\sum_{j=1}^n (x_k-y_k)\frac{\partial}{\partial y_k}\right) \circ Q\left( \frac{ \partial}{\partial \mathbf x}\right)
\]
\[+2(-\frac{n}{2} +1+\lambda) \left(\sum_{j=1}^n (y_k-x_k) \frac{\partial}{\partial x_k}\right) \circ Q\left( \frac{ \partial}{\partial \mathbf y}\right)
\]
\[-2\mu (-\frac{n}{2} +1+\mu)\ Q\left( \frac{\partial}{\partial \mathbf x}\right)-2 \lambda (-\frac{n}{2} +1+\lambda)\ Q\left( \frac{\partial}{\partial \mathbf y}\right)
\]
\[+4(-\frac{n}{2}+1+\lambda)(-\frac{n}{2}+1+\mu)\ Q\left(\frac{\partial}{\partial \mathbf x}, \frac{\partial}{\partial \mathbf y} \right)\ .
\]
\begin{proposition} Let $f\in \mathcal H_{(\lambda, \epsilon),(\mu,\eta)}$. Extend $f$ to a neighborhood of $\Xi\times \Xi$, preserving the homogeneities. Then the restriction of $\mathbf F_{\lambda, \mu} f$ to $\Xi\times \Xi$  does not depend on the particular extension of $f$ and is equal to $\mathbf D_{(\lambda,\epsilon),(\mu, \eta)}^{reg}f$.
\end{proposition}

\begin{proof} Use \eqref{modifIX} and the homogeneities of the extension of $f$ to prove that
$\mathbf F_{\lambda, \mu} f$ coincides with $\mathbf E_{\lambda, \mu}^{reg} f$.
\end{proof}

\noindent
{\bf Remark.} The operator $\mathbf F_{\lambda, \mu}$ exhibits  symmetry with respect to the couples $(\mathbf x\leftrightarrow \mathbf y),(\lambda\leftrightarrow\mu)$ which did not exist for the initial operator $\mathbf E_{\lambda, \mu}$. In particular, it is possible to start with the operator
\[\sgn(Q(\mathbf x, \mathbf y)) \quad \times\]\[\vert Q(\mathbf x,\mathbf y)\vert^{+\frac{n}{2}-\lambda-1}\circ \square\left(\frac{\partial}{\partial \mathbf x}\right)\circ \vert Q(\mathbf x,\mathbf y)\vert^{\lambda-\mu} \circ \square\left(\frac{\partial}{\partial \mathbf y}\right)\circ \vert Q(\mathbf x, \mathbf y)\vert^{-\frac{n}{2}+2+\mu}\ ,
\]
which is \emph{not} equal to $\mathbf E_{\lambda, \mu}$. However, the  process of regularization produces the \emph{same} differential operator $\mathbf D_{(\lambda, \epsilon),(\mu,\eta)}^{reg}$ on $X\times X$.

\footnotesize{ \noindent Address\\  Institut \'Elie Cartan, Universit\'e de Lorraine 54506 Vand\oe uvre-l\`es Nancy (France)
\smallskip

\noindent \texttt{{jean-louis.clerc@univ-lorraine.fr 
}}

 \end{document}